\documentclass[11pt]{amsart}
\pdfoutput=1

  \pdfpageattr{/Group <</S /Transparency /I true /CS /DeviceRGB>>}

\usepackage[T1]{fontenc}
\usepackage[english]{babel}
\usepackage{geometry}
\usepackage{amsfonts,amsmath,amssymb,amsthm,mathrsfs}
\usepackage{caption}
\usepackage[labelfont=rm]{subcaption}
\usepackage{xspace}
\usepackage{marvosym}
\usepackage{microtype}
\usepackage{enumerate}
\usepackage{mathtools}
\usepackage{booktabs}
\usepackage{url}

\microtypesetup{tracking,kerning,spacing}
\microtypecontext{spacing=nonfrench}

\newcommand\CC{{\mathbb C}}
\newcommand\NN{{\mathbb N}}

\newcommand\KK{{\mathbb K}}

\newcommand\QQ{{\mathbb Q}}
\newcommand\RR{{\mathbb R}}

\newcommand\cGB{G} 
\newcommand\uGB{U}
\newcommand\cTB{T} 

\newcommand\SetOf[2]{\left\{\left.#1\vphantom{#2}\ \right|\ #2\vphantom{#1}\right\}}

\DeclareMathOperator{\val}{val}

\DeclareMathOperator{\variety}{V}
\newcommand\trop{{\mathcal T}} 
\DeclareMathOperator{\initial}{in}

\theoremstyle{plain}
    \newtheorem{theorem}{Theorem}
    \newtheorem{corollary}[theorem]{Corollary}
    
    \newtheorem{proposition}[theorem]{Proposition}
\theoremstyle{definition}
    \newtheorem{remark}[theorem]{Remark}
    \newtheorem{example}[theorem]{Example}

    \newtheorem{question}{Question}

\title{The degree of a tropical basis}

\author{Michael Joswig \and Benjamin Schr\"oter}

\address[Michael Joswig, Benjamin Schr\"oter]{
Institut f{\"u}r Mathematik,
 TU Berlin,
 Str.\ des 17. Juni 136, 10623 Berlin, Germany
}
\email{\{joswig,schroeter\}@math.tu-berlin.de}

\thanks{M.~Joswig is supported by Einstein Foundation Berlin and Deutsche Forschungsgemeinschaft (DFG)}

\subjclass[2010]{}

\begin{document}

\begin{abstract}
  We give an explicit upper bound for the degree of a tropical basis of a homogeneous polynomial ideal.
  As an application $f$-vectors of tropical varieties are discussed.
  Various examples illustrate differences between Gr\"obner and tropical bases.
\end{abstract}
\maketitle

\section{Introduction}
\noindent
Computations with ideals in polynomial rings require an explicit representation in terms of a finite set of polynomials which generate that ideal.
The size, i.e., the amount of memory required to store this data, depends on three parameters: the number of generators, their degrees, and the sizes of their coefficients.
For purposes of computational complexity it is of major interest to study for which parameters generating sets exist.
An early step in this direction is Hermann's degree bound \cite{Hermann:1926} on solutions of linear equations over $\QQ[x_1,\dots,x_n]$.
In practice, however, not all generating sets are equally useful, and so it is important to seek complexity results for generating sets which have additional desirable properties.
A landmark result here is the worst case space complexity estimate for Gr\"obner bases by Mayr and Meyer \cite{MayrMeyer:1982}.

Tropical geometry associates with an algebraic variety a piecewise linear object in the following way.
Let $\KK$ be a field with a real-valued valuation, which we denote as $\val$.
We consider an ideal $I$ in the polynomial ring $\KK[x_1,\dots,x_n]$ and its vanishing locus $\variety(I)$, which is an affine variety.
The \emph{tropical variety} $\trop(I)$ is defined as the topological closure of the set
\begin{equation}
  \val\bigl(\variety(I)\bigr) \ = \ \SetOf{\strut\bigl(\val(z_1),\dots,\val(z_n)\bigr)}{z\in\variety(I)\cap (K\setminus\{0\})^n} \quad \subset \ \RR^n \enspace .
\end{equation}
In general, $\trop(I)$ is a polyhedral complex whose dimension agrees with the Krull dimension of $I$; see Bieri and Groves \cite{BieriGroves:1984}.
If, however, the ideal $I$ has a generating system of polynomials whose coefficients are mapped to zero by $\val$ that polyhedral complex is a fan.
This is the \emph{constant coefficient case}.
A major technical challenge in tropical geometry is the fact that, in general, intersections of tropical varieties do not need to be tropical varieties.
Therefore, the following concept is crucial for an approach via computational commutative algebra.
A finite generating subset $\cTB$ of $I$ is a \emph{tropical basis} if the tropical variety $\trop(I)$ coincides with the intersection of the finitely many tropical hypersurfaces $\trop(f)$ for $f\in \cTB$.
See \cite[Def.~2.6.3]{MaclaganSturmfels:2015} for an equivalent definition in terms of initial ideals.

Assuming constant coefficients, our main result states that each such ideal has a tropical basis whose degree does not exceed a certain bound which is given explictly.
While the bound which we are currently able to achieve is horrendous, to the best of our knowledge this is the first result of this kind.
Moreover, we present examples of tropical bases which exhibit several interesting features.
We close this paper with an application to $f$-vectors of tropical varieties and two open problems.

\section{Degree bounds}
\noindent
Throughout the following let $I$ be a homogeneous ideal in the polynomial ring $R:=K[x_1,\dots,x_n]$.
Bogart et al.\ were the first to describe an algorithm for computing a tropical basis \cite[Thm.~11]{BogartJensenSpeyerSturmfels:2007}. 
This algorithm is implemented in \texttt{Gfan}, a software package for computing Gr\"obner fans and tropical varieties \cite{gfan}.
Further, in \cite{Ren:2015}, Ren presents an enhanced algorithm which also covers non-constant coefficients.
Since our proof rests on the method of Bogart et al.\, we need to give a few more details.
Every weight vector $w\in\RR^n$ gives rise to a generalized term order on $R$.
The generalization lies in the fact that this order may only be partial, which is why the initial form $\initial_w(f)$ of a polynomial $f$ does not need to be a monomial.
Now the tropical variety of $I$ can be described as the set
\[
\trop(I) \ = \ \SetOf{w\in\RR^n}{\initial_w(I) \text{ does not contain any monomial}} \enspace ,
\]
where the \emph{initial ideal} $\initial_w(I)$ is generated from all initial forms of polynomials in $I$.
Declaring two weight vectors equivalent whenever their initial ideals agree yields a stratification of $\RR^n$ into relatively open polyhedral cones;
this is the \emph{Gr\"obner fan} of $I$.
Each maximal cone of the Gr\"obner fan corresponds to a proper term order or, equivalently, to a monomial initial ideal and a reduced Gr\"obner basis.
A Gr\"obner basis is \emph{universal} if it is a Gr\"obner basis for each term order.
By construction $\trop(I)$ is a subfan of the Gr\"obner fan.
A polynomial $f\in I$ is a \emph{witness} for a weight vector $w\in\RR^n$ if its initial form $\initial_w(f)$ is a monomial.
Such an $f$ certifies that the Gr\"obner cone containing $w$ is not contained in $\trop(I)$.
The algorithm in \cite{BogartJensenSpeyerSturmfels:2007} now checks each Gr\"obner cone and adds witnesses to a universal Gr\"obner basis to obtain a tropical basis.

To shorten our notation we fix the vector $w$ and abbreviate $J:=\initial_w(I)$.
The ideal $J$ contains the monomial $x^m=x_1^{m_1}\cdots x_n^{m_n}$ if and only if the \emph{saturation}
\[
J:x^m \ = \ \SetOf{f\in R}{x^m f \in J}
\]
contains a unit.
Hence the total degree of any witness does not exceed $\alpha n$, where $\alpha$ is the maximal saturation exponent of all initial ideals of $I$ with respect to $x_1\cdots x_n$.
We need to get a grip on that parameter $\alpha$.
The \emph{degree} of a finite set of polynomials is the maximal total degree which occurs.

\begin{proposition}\label{prop:saturation}
  Let $J$ be a homogeneous ideal.
  The saturation exponent $\alpha$ of $J$ with respect to $x_1\cdots x_n$ is bounded by 
  \[
  \alpha \ \leq \ \deg H \enspace ,
  \]
  where $H$ is a universal Gr\"obner basis for $J$.
\end{proposition}
\begin{proof}
  Since $H$ is universal it contains a Gr\"obner basis $\{f_1,\ldots,f_s\}$ for the reverse lexicographic order.
  By \cite[Prop.~15.12]{Eisenbud:1995} the set
  \[
  \left\{\frac{f_1}{\gcd(x_n,f_1)},\frac{f_2}{\gcd(x_n,f_2)},\ldots,\frac{f_s}{\gcd(x_n,f_s)}\right\}
  \]
  is a Gr\"obner basis for $J:x_n$.
  Thus the saturation exponent of $J$ with respect to $x_n$ is bounded by the degree $\deg_{x_n}(H)$ of $H$ in the variable $x_n$.
  Permuting the variables implies a similar statement for $x_i$.
  It follows that $\alpha = \max_{1\leq i \leq n} \deg_{x_i} H \ \leq \ \deg H$.
\end{proof}

Notice that the tropical variety of a homogeneous ideal $I$ coincides with the tropical variety of the saturated ideal
$
I:(x_1\cdots x_n)^\infty = \bigcup_{k\in\NN} I:(x_1\cdots x_n)^k
$.
For the next step we need to determine the degree of a universal Gr\"obner basis.
The key ingredient is a result of Mayr and Ritscher \cite{MayrRitscher:2010}.
Here and below $d$ is the \emph{degree} of $I$, i.e., the minimum of the degrees of all generating sets, and $r$ is the Krull dimension.

\begin{proposition}[Mayr and Ritscher]\label{prop:universal}
  Assume that $r \geq 1$.
  Each reduced Gr\"obner basis $\cGB$ of the ideal $I$ satisfies
  \begin{equation}\label{eq:GB_degree}
    \deg\cGB \ \leq \ 2 \left(\frac{d^{n-r}+d}{2}\right)^{2^{r-1}} \enspace .
  \end{equation}
\end{proposition}
Lakshman and Lazard \cite{LakshmanLazard:1991} give an asymptotic bound of the degree on zero-dimensional ideals, that is, for $r=0$.
For Gr\"obner bases one could argue that the degree is more interesting than the number of polynomials.
This is due to the following simple observation.

\begin{remark}
  A reduced Gr\"obner basis of degree $e$ (of any ideal in $R$) can contain at most
  $\tbinom{e+n-1}{e}=\tbinom{e+n-1}{n-1}$ polynomials.  The reason is that no two leading monomials
  can divide one another.
\end{remark}

We are ready to bound the degree of a universal Gr\"obner basis.
In view of the previous remark this also entails a bound on the number of polynomials.
Since we will use Proposition~\ref{prop:universal}, throughout this section we will assume that $r\geq 1$.

\begin{corollary}\label{cor:universal}
  There is a universal Gr\"obner basis for $I$ whose degree is bounded by \eqref{eq:GB_degree}.
\end{corollary}
\begin{proof}
  The union of the reduced Gr\"obner bases for all term orders is universal.
  The claim follows since the bound in Proposition~\ref{prop:universal} is uniform.
\end{proof}

For our main result we apply the bounds which we just obtained to the output of the algorithm in \cite{BogartJensenSpeyerSturmfels:2007}.

\begin{theorem}\label{thm:main}
  Suppose that the valuation $\val$ on the coefficients is trivial.
  There is a tropical basis $\cTB$ of the homogenous ideal $I$ with
  \begin{equation}\label{eq:TB_degree}
    \deg\cTB \ \leq \ \max\left\{\deg\uGB,\alpha n\right\} \ \leq \ n \deg\uGB \ \leq \ 2 n \left(\frac{d^{n-r}+d}{2}\right)^{2^{r-1}} \enspace ,
  \end{equation}
  where $\uGB$ is a universal Gr\"obner basis for $I$.
\end{theorem}
\begin{proof}
  The number $\alpha n$ bounds the degree of a witness, and so the first inequality follows from the correctness of the algorithm \cite[Thm.~11]{BogartJensenSpeyerSturmfels:2007}.
  From $\uGB$ we can obtain a universal Gr\"obner basis $H$ for any initial ideal $J$, and this satisfies $\deg H \leq \deg\uGB$.
  An initial ideal of $J$ with term ordering $w$ coincides with the initial ideal of $I$ with respect to a perturbation of the term order that yields $J$ in direction $w$.
  From Proposition~\ref{prop:saturation} we thus get the second inequality.
  Finally, the third inequality follows from \eqref{eq:GB_degree} and Corollary~\ref{cor:universal}.
\end{proof}

\section{Examples}
\noindent
Throughout this section, we will be looking at the case $\KK=\CC$, and $\val$ sends each non-zero complex number to zero.
In particular, as above, we are considering constant coefficients.

It is known that, in general, a universal Gr\"obner basis does not need to be a tropical basis; see \cite[Ex.~10]{BogartJensenSpeyerSturmfels:2007} or \cite[Ex.~2.6.7]{MaclaganSturmfels:2015}.
That is, it cannot be avoided to compute witness polynomials.
In fact, the following example, which is a simple modification of \cite[Ex.~10]{BogartJensenSpeyerSturmfels:2007}, shows that adding witnesses may even increase the degree.

\begin{example}\label{ex:universal}
  Let $I\subset\CC[x,y,z]$ be the ideal generated by the six degree $3$ polynomials
  \begin{center}
    \begin{tabular}{ccc}
      $x^2 y + x y^2\,$, & $x^2 z + x z^2\,$, & $y^2 z + y z^2$\,,\\
      $x^3 + x^2y + x^2 z\,$, & $xy^2 + y^3 + y^2z\,$, & $x  z^2 + yz^2 + z^3$\\
    \end{tabular}
  \end{center}
  These six generators together with the ten polynomials of degree $3$ below form a universal Gr\"obner basis for $I$.
  \begin{center}
    \begin{tabular}{ccc}
    $x^3 - x y^2 - x z^2$\,, & $x^2  y - y^3 + y z^2$\,, & $x^2  z + y^2  z - z^3$\,, \\
    $x^3 - x y^2 + x^2 z$\,, & $x  y^2 + y^3 - y z^2$\,, & $x  z^2 - y^2  z + z^3$\,, \\
    $x^3 + x^2 y - x z^2$\,, & $x^2  y - y^3 - y^2 z$\,, & $x^2  z - y  z^2 - z^3$\,, \\ 
    & $x^3+y^3+z^3$ &\\
  \end{tabular}
  \end{center}
  The monomial $x^2yz$ of degree $4$ is contained in $I$.
  This is a witness to the fact that the tropical variety $\trop(I)$ is empty.
  Since, however, there is no monomial of degree $3$ contained in $I$, any tropical basis must have degree at least $4$.
  One such tropical basis, $\cTB$, is given by the six generators and the monomial $x^2yz$.
  This also shows that a tropical basis does not need to contain a universal Gr\"obner basis.
\end{example}

A tropical basis does not even need to be any Gr\"obner basis, as the next example shows.
\begin{example}\label{ex:not-groebner}
  Consider the three polynomials
  \[ x^5 \, , \quad x^4+x^2y^2+y^4 \, , \quad y^5 \]
  in $\CC[x,y]$.
  They form a tropical basis for the ideal they generate. However each Gr\"obner basis has to include at least one of the S-polynomials $x^3y^2+xy^4$ or $x^4y+x^2y^3$.
\end{example}

For conciseness the Examples~\ref{ex:universal} and~\ref{ex:not-groebner} address tropical varieties which are empty.
One can modify the above to obtain ideals and systems of generators with similar properties for tropical varieties of arbitrarily high dimension.
We leave the details to the reader.

It is obvious that the final upper bound in \eqref{eq:TB_degree} is an extremely coarse estimate.
However, better bounds on the degree of the universal Gr\"obner basis can clearly be exploited.
The following example may serve as an illustration.

\begin{example}\label{ex:Delta24}
  Let $I=\langle xy-zw+uv\rangle\subset\CC[x,y,z,u,v,w]$. In this case we have $d=2$, $n=6$ and $r=5$.
  Since $I$ is a principal ideal the single generator forms a Gr\"obner basis, which is even universal and also a tropical basis.
  The degree of that universal Gr\"obner basis is $d=2$, which needs to be compared with the upper bound of $2^{17}$ from \eqref{eq:GB_degree}.
  For the saturation exponent we have $\alpha=1\leq 2$, and the degree of the tropical basis equals $d=2$.
  This is rather close to the bound $\alpha n=6$, whereas the final upper bound in \eqref{eq:TB_degree} is as much as $3\cdot 2^{18}$.
\end{example}

Our final example generalizes the previous.
In fact, Example~\ref{ex:Delta24} re-appears below for $D=2$ and $N=4$.

\begin{example}
  The Pl\"ucker ideal $I_{D,N}$ captures the algebraic relations among the $D{\times}D$-minors of a generic $D{\times}N$-matrix with coefficients in the field $\KK$.
  This is a homogeneous prime ideal in the polynomial ring over $\KK$ with $n=\tbinom{N}{D}$ variables.
  The variety  $\variety(I_{D,N})$ is the \emph{Grassmannian} of $D$-planes in $\KK^N$.
  Its tropicalization $\trop(I_{D,N})$ is the \emph{tropical Grassmannian} of Speyer and Sturmfels \cite{SpeyerSturmfels:2004}; see also \cite[\S4.3]{MaclaganSturmfels:2015}.

  The Pl\"ucker ideal is generated by quadratic relations; see \cite[Thm~3.1.7]{Sturmfels:2008}.  
  Its dimension equals $r=(N-D)D+1$; see \cite[Cor~3.1]{SpeyerSturmfels:2004}.
  From this data we derive that there is a tropical basis $\cTB_{D,N}$ of degree
  \[
  \deg\cTB_{D,N} \ \leq \ 2 \cdot \tbinom{N}{D} \cdot \left(2^{\tbinom{N}{D}-ND+D^2-2}+1\right)^{2^{ND-D^2}} \enspace .
  \]
  To the best of our knowledge explicit tropical bases for $I_{D,N}$ are known only for $D=2$ and $(D,N)\in\{(3,6),\,(3,7)\}$; see \cite{SpeyerSturmfels:2004} and \cite{HerrmannJensenJoswigSturmfels:2009}.
  Note that for $D=2$ the degree of a universal Gr\"obner basis grows with $n$ while the quadratic $3$-term Pl\"ucker relations form a tropical basis.
\end{example}

\section{The $f$-vector of a tropical variety}
\noindent
The \emph{$f$-vector} of a polyhedral complex, which counts the number of cells by dimension, is a fundamental combinatorial complexity measure.
In this section we will give an explicit bound on the $f$-vector of a tropical variety $\trop(I)$, with arbitrary valuation on the field $\KK$, in terms of the number $s$ of polynomials in a tropical basis $\cTB$ and the degree $d$ of a tropical basis, $\cTB$.
Notice that in the previous sections ``$d$'' was the degree of $I$.

First we discuss the case of a tropical hypersurface, that is, $s=1$, as in Example~\ref{ex:Delta24}.
Let $g\in R$ be an arbitrary homogeneous polynomial of degree $d$.
In contrast to the previous sections, here we are admitting non-constant coefficients.
A tropical hypersurface $\trop(g)$ is dual to the regular subdivision of the Newton polytope $N(g)$ of $g$, which is gotten from lifting the lattice points in $N(g)$, which correspond to the monomials in $g$, to the valuation of their coefficients \cite[Prop.~3.1.6]{MaclaganSturmfels:2015}.
See the monograph \cite{LoeraRambauSantos:2010} for details on polytopal subdivisions of finite point sets.
The polynomial $g$ has at most $\tbinom{d+n-1}{n-1}$ monomials, which correspond to the lattice points in the $d$th dilation of the $(n-1)$-dimensional simplex $d\cdot\Delta_{n-1}$.
The \emph{standard simplex} $\Delta_{n-1}$ is the $(n{-}1)$-dimensional convex hull of the $n$ standard basis vectors $e_1,\ldots,e_n$.
The maximal $f$-vector of a polytopal subdivision of $d\cdot \Delta_{n-1}$ by lattice points is (simultaneously for all dimensions) attained for a unimodular triangulation \cite[Thm.~2]{BetkeMcMullen:1985}.
If $\Delta$ is such a unimodular triangulation, then its vertices use all lattice points in $d\cdot\Delta_{n-1}$.
The converse does not hold if $n\geq 4$.
The $f$-vector of $\Delta$ equals
\begin{equation}\label{eq:f-simplex}
  f^\Delta_j \ = \ \sum_{i=0}^j (-1)^{i+j} \tbinom{j}{i}\tbinom{di+d+n-1}{n-1} \enspace ;
\end{equation}
see \cite[Thm.~9.3.25]{LoeraRambauSantos:2010}.
By duality the bound in \eqref{eq:f-simplex} translates into a bound on the $f$-vector for the tropical hypersurface $\trop(g)$:
\begin{equation}\label{eq:f-hypersurface}
  f^{\trop(g)}_j \ \leq \  f^\Delta_{n-j-1} \ \leq \ \sum_{i=1}^{n-j} (-1)^{n+i-j} \tbinom{n-j-1}{i-1}\tbinom{di+n-1}{n-1} \enspace .
\end{equation}

From the above computation we can derive the following general result.
\begin{proposition}
  Let $I$ be a homogeneous ideal in $R$.
  Then the $f$-vector of the tropical variety $\trop(I)$ with a tropical basis $\cTB$, consisting of $s$ polynomials of degree at most $d$, is bounded by
  \[
  f_j \ \leq \  \sum_{i=1}^{n-j} (-1)^{n+i-j} \tbinom{n-j-1}{i-1}\tbinom{sdi+n-1}{n-1} \enspace .
  \]
\end{proposition}
\begin{proof}
Let $g$ denote the product $h_1\cdots h_s$ of all polynomials in the tropical basis $\cTB$.
The tropical hypersurface of $g$ is the support of the $(n{-}1)$-skeleton of the polyhedral complex dual to a regular subdivision of the Newton polytope $N(g)$; see \cite[Prop.~3.1.6]{MaclaganSturmfels:2015}.
This polytope is the Minkowski sum of all Newton polytopes $N(h)$ for $h\in\cTB$. 
Moreover, the polyhedral subdivision of $N(g)$ dual to $\trop(g)$ is the common refinement of the subdivisions of the Newton polytopes for the polynomials in $\cTB$.
The tropical variety $\trop(I)$ is a subcomplex of this refinement since, by the definition of $T$, we have
\[
\trop(I) \ = \ \bigcap_{f\in T} \trop(f) \enspace .
\]
The polynomial $g$ is of degree at most $sd$.
From the inequality~\eqref{eq:f-hypersurface} we get the claim.
\end{proof}

Let us now discuss the special case of a tropical hypersurface $\trop(g)$ with constant coefficients.
That is, we assume that the valuation map applied to each coefficient of the homogeneous polynomial $g$ yields zero.
In this case the lifting is trivial and thus $\trop(g)$ is dual to a lattice polytope contained in the simplex $d\cdot\Delta_{n-1}$; see \cite[Prop.~3.1.10]{MaclaganSturmfels:2015}.
We introduce the parameter
\[
\lambda_j(d,n) \ = \ \max \SetOf{f^P_j}{\text{$P$ is a lattice polytope in $d\cdot\Delta_{n-1}$}} \enspace,
\]
which measures how combinatorially complex tropical hypersurfaces (with constant coefficients) can be.
We arrive at the following conclusion.
\begin{corollary}
  Let $I$ be a homogeneous ideal in $R$ which is generated by polynomials with constant coefficients.
  Then the $f$-vector of the tropical variety $\trop(I)$ a tropical basis $\cTB$, consisting of $s$ polynomials of degree at most $d$, is bounded by
  \begin{equation}\label{eq:f-tropvariety-const}
    f_j \ \leq \ \lambda_{n-j-1}(sd,n) \enspace .
  \end{equation}
\end{corollary}

Notice that the $(n-1)$-simplex has $\lambda_0(1,n)=n$ vertices and an interval has $\lambda_0(d,2)=2$ vertices. The number of vertices $\lambda_0(d,n)$ does not exceed the sum of the number of vertices in $(d-1)\cdot\Delta_{n-1}$ and $d\cdot\Delta_{n-2}$.
Hence, e.g., the number of $(n{-}1)$-cells of $\trop(I)$ in \eqref{eq:f-tropvariety-const} is bounded by
\[
  f_{n-1}\ \leq \ \lambda_0(sd,n) \ \leq \ \sum_{i=0}^{sd-2}2\binom{i+n-3}{i}+\sum_{i=0}^{n-3}(n-i)\binom{i+sd-2}{i} \enspace .
\]

\begin{table}[t]
\centering
\caption{The vectors $\lambda(d,n)$ for small values of $d$ and $n$.
  A star indicates only a lower bound, which is due to the fact that we could not complete our ad hoc computation with the given resources.}
\label{tab:lambda}
\begin{tabular*}{1.00\linewidth}{@{\extracolsep{\fill}}cccccc@{}}
    \toprule
   $n \backslash d$ & $1$ & $2$ & $3$ & $4$ & $5$\\
     \midrule
     $2$ & $(2)  $       & $(2)$            & $(2)$             & $(2)$             & $(2)$\\ 
     $3$ & $(3,3)$       & $(4,4)$          & $(6,6)$           & $(6,6)$           & $(8,8)$\\
     $4$ & $(4,6,4)$     & $(7,12,8)$       & $(12,18,10)$      & $(15,24,16)^*$    & $(18,33,18)^*$\\
     $5$ & $(5,10,10,5)$ & $(11,30,30,10)$  & $(20,48,47,20)^*$ & $(25,70,73,26)^*$ & $(31,82,77,25)^*$\\
     \bottomrule
  \end{tabular*}
\end{table}

We calculated the numbers $\lambda_j(d,n)$ for small values of $d$ and $n$ with \texttt{polymake} \cite{DMV:polymake}.
The result is summarized in Table~\ref{tab:lambda}.
Note that, e.g., for $d=2$ and $n=4$ there is no polytope that maximizes $f_j$ simultaneously for all $j$.
We expect that it is difficult to explicitly determine the values for $\lambda_j(d,n)$.
The somewhat related question of determining the (maximal) $f$-vectors of $0/1$-polytopes is a challenging open problem; see~\cite{Ziegler:2000}.

\section{Open Problems}
\noindent
Our results immediately raise a number of questions.
We mention two, which we find most interesting.

\begin{question}
  How can the upper bound \eqref{eq:TB_degree} in Theorem~\ref{thm:main} for the degree of a tropical basis in the constant coefficient case be generalized to cover arbitrary coefficients?
\end{question}
This amounts to replacing the analysis of \cite[Thm.~11]{BogartJensenSpeyerSturmfels:2007} in our proof of Theorem~\ref{thm:main} by scrutinizing the more involved general methods in \cite[Thm.~2.6.6]{MaclaganSturmfels:2015} or \cite{Ren:2015}.
Another approach would be via the alternative algorithm of Hept and Theobald \cite{HeptTheobald:2009}, which is based on projections.

Our current techniques (for constant coefficients) employ the Gr\"obner fan of an ideal, i.e., a universal Gr\"obner basis.
Yet, as Example~\ref{ex:not-groebner} shows tropical bases and Gr\"obner bases are not related in a straightforward way.
\begin{question}
  Is it possible to directly obtain a tropical basis from the generators of an ideal, i.e., without the need to compute any Gr\"obner basis?
\end{question}
Notice that the algorithm of Hept and Theobald \cite{HeptTheobald:2009} uses elimination (Gr\"obner bases).
However, one may ask if techniques from polyhedral geometry can further be exploited to obtain yet another method for computing tropical bases and tropical varieties.

\bibliographystyle{alpha}
\bibliography{References}

\newcommand{\etalchar}[1]{$^{#1}$}
\begin{thebibliography}{DLRS10}

\bibitem[BG84]{BieriGroves:1984}
Robert Bieri and J.~R.~J. Groves.
\newblock The geometry of the set of characters induced by valuations.
\newblock {\em J. Reine Angew. Math.}, 347:168--195, 1984.

\bibitem[BJS{\etalchar{+}}07]{BogartJensenSpeyerSturmfels:2007}
Tristram Bogart, Anders~N. Jensen, David Speyer, Bernd Sturmfels, and Rekha~R.
  Thomas.
\newblock Computing tropical varieties.
\newblock {\em J. Symbolic Comput.}, 42(1-2):54--73, 2007.

\bibitem[BM85]{BetkeMcMullen:1985}
Ulrich Betke and Peter McMullen.
\newblock Lattice points in lattice polytopes.
\newblock {\em Monatsh. Math.}, 99(4):253--265, 1985.

\bibitem[DLRS10]{LoeraRambauSantos:2010}
Jes{\'u}s~A. De~Loera, J{\"o}rg Rambau, and Francisco Santos.
\newblock {\em Triangulations}, volume~25 of {\em Algorithms and Computation in
  Mathematics}.
\newblock Springer-Verlag, Berlin, 2010.
\newblock Structures for algorithms and applications.

\bibitem[Eis95]{Eisenbud:1995}
David Eisenbud.
\newblock {\em Commutative algebra}, volume 150 of {\em Graduate Texts in
  Mathematics}.
\newblock Springer-Verlag, New York, 1995.
\newblock With a view toward algebraic geometry.

\bibitem[GJ00]{DMV:polymake}
Ewgenij Gawrilow and Michael Joswig.
\newblock \polymake: a framework for analyzing convex polytopes.
\newblock In {\em Polytopes---combinatorics and computation (Oberwolfach,
  1997)}, volume~29 of {\em DMV Sem.}, pages 43--73. Birk\-h\"au\-ser, Basel,
  2000.

\bibitem[Her26]{Hermann:1926}
Grete Hermann.
\newblock Die {F}rage der endlich vielen {S}chritte in der {T}heorie der
  {P}olynomideale.
\newblock {\em Math. Ann.}, 95(1):736--788, 1926.

\bibitem[HJJS09]{HerrmannJensenJoswigSturmfels:2009}
Sven Herrmann, Anders Jensen, Michael Joswig, and Bernd Sturmfels.
\newblock How to draw tropical planes.
\newblock {\em Electron. J. Combin.}, 16(2, Special volume in honor of Anders
  Bj\"orner):\ Research Paper~6, 26, 2009.

\bibitem[HT09]{HeptTheobald:2009}
Kerstin Hept and Thorsten Theobald.
\newblock Tropical bases by regular projections.
\newblock {\em Proc. Amer. Math. Soc.}, 137(7):2233--2241, 2009.

\bibitem[Jen]{gfan}
Anders~N. Jensen.
\newblock {G}fan, a software system for {G}r{\"o}bner fans and tropical
  varieties.
\newblock Available at
  \url{http://home.imf.au.dk/jensen/software/gfan/gfan.html}.

\bibitem[LL91]{LakshmanLazard:1991}
Y.~N. Lakshman and Daniel Lazard.
\newblock On the complexity of zero-dimensional algebraic systems.
\newblock In {\em Effective methods in algebraic geometry ({C}astiglioncello,
  1990)}, volume~94 of {\em Progr. Math.}, pages 217--225. Birkh\"auser Boston,
  Boston, MA, 1991.

\bibitem[MM82]{MayrMeyer:1982}
Ernst~W. Mayr and Albert~R. Meyer.
\newblock The complexity of the word problems for commutative semigroups and
  polynomial ideals.
\newblock {\em Adv. in Math.}, 46(3):305--329, 1982.

\bibitem[MR10]{MayrRitscher:2010}
Ernst~W. Mayr and Stephan Ritscher.
\newblock Degree bounds for {G}r\"obner bases of low-dimensional polynomial
  ideals.
\newblock In {\em I{SSAC} 2010---{P}roceedings of the 2010 {I}nternational
  {S}ymposium on {S}ymbolic and {A}lgebraic {C}omputation}, pages 21--27. ACM,
  New York, 2010.

\bibitem[MS15]{MaclaganSturmfels:2015}
Diane Maclagan and Bernd Sturmfels.
\newblock {\em Introduction to {T}ropical {G}eometry}, volume 161 of {\em
  Graduate Studies in Mathematics}.
\newblock American Mathematical Society, Providence, RI, 2015.

\bibitem[Ren15]{Ren:2015}
Yue Ren.
\newblock {\em Tropical geometry in \textsc{Singular}}.
\newblock PhD thesis, Technische Universit\"at Kaiserslautern, Germany, 2015.

\bibitem[SS04]{SpeyerSturmfels:2004}
David Speyer and Bernd Sturmfels.
\newblock The tropical {G}rassmannian.
\newblock {\em Adv. Geom.}, 4(3):389--411, 2004.

\bibitem[Stu08]{Sturmfels:2008}
Bernd Sturmfels.
\newblock {\em Algorithms in invariant theory}.
\newblock Texts and Monographs in Symbolic Computation. Springer-Verlag,
  Vienna, second edition, 2008.

\bibitem[Zie00]{Ziegler:2000}
G{\"u}nter~M. Ziegler.
\newblock Lectures on {$0/1$}-polytopes.
\newblock In {\em Polytopes---combinatorics and computation ({O}berwolfach,
  1997)}, volume~29 of {\em DMV Sem.}, pages 1--41. Birkh\"auser, Basel, 2000.

\end{thebibliography}
\end{document}